\documentclass[11pt]{article}
 
\usepackage[margin=0.5in]{geometry} 
\usepackage{amsmath,amsthm,amssymb}
\usepackage{mathrsfs}
\usepackage[T1]{fontenc}
\usepackage[english]{babel}
\usepackage{graphicx}
\usepackage{color} 

\usepackage{hyperref}
\usepackage{afterpage}

\usepackage{color}

\usepackage{mathtools}
\usepackage{atbegshi}
\AtBeginDocument{\AtBeginShipoutNext{\AtBeginShipoutDiscard}}
\usepackage{pdfpages} 
\usepackage{longtable} 
\usepackage{amsmath,amsthm,amsfonts,amssymb,relsize,bm} 
\usepackage{xparse,xcolor} 
\usepackage{graphicx,float} 
\usepackage{etoolbox}
\usepackage[shortlabels]{enumitem} 
\usepackage{tikz-cd} 
\usepackage{cite} 
\usepackage[normalem]{ulem}
\usepackage{comment}
\usepackage{tocloft}

\numberwithin{equation}{section}

\theoremstyle{plain}
\newtheorem{theorem}{Theorem}[section]

\newtheorem*{theorem*}{Theorem}
\newtheorem{lemma}{Lemma}[section]

\theoremstyle{definition}

\newtheorem*{definition*}{Definition}

\begin{document}
 
 \title{An analogue of a formula of Popov}

\author{{Pedro Ribeiro}} 
\thanks{


Department of Mathematics, Faculty of Sciences of University of Porto, Rua do Campo Alegre,  687; 4169-007 Porto (Portugal). 

\,\,\,\,\,E-mail of the corresponding author: pedromanelribeiro1812@gmail.com}

\date{}
 
\maketitle

\begin{abstract} Let $r_{k}(n)$ denote the number of representations of the positive
integer $n$ as the sum of $k$ squares. We prove a new summation formula involving $r_{k}(n)$ and the Bessel functions of the first kind, which constitutes an
analogue of a result due to the Russian mathematician
A. I. Popov.
\end{abstract}

\pagenumbering{arabic}

\section{Introduction}
The insights contained in Ramanujan's lost notebook \cite{Ramanujan_lost} still
inspire a lot of new mathematical ideas. Recently, the discovery of
some overlooked works from the Russian school, namely from the mathematician
N. S. Koshliakov \cite{KOSHLIAKOV} has produced a similar effect on
mathematical research (cf. \cite{DG, Ramanujan_meets, Kosh_Petronilho}). 
Belonging to the same tradition as Koshliakov, the mathematician and
linguist Alexander Ivanovich Popov (1899-1973) made his contribution
to Mathematics by providing new interesting summation formulas. 

In a fairly unknown paper \cite{popov_russian}, Popov states the following beautiful result. If $\text{Re}(x)>0$ and $z\in\mathbb{C}$, then
\begin{align}
\frac{z^{\frac{k}{2}-1}\pi^{\frac{k}{4}-\frac{1}{2}}x^{\frac{k}{4}}}{2^{\frac{k}{2}-1}\Gamma\left(\frac{k}{2}\right)}\,e^{z^{2}/8}+\sqrt{x}\,e^{z^{2}/8}\,\sum_{n=1}^{\infty}\frac{r_{k}(n)}{n^{\frac{k}{4}-\frac{1}{2}}}\,e^{-\pi nx}\,J_{\frac{k}{2}-1}(\sqrt{\pi nx}z)\nonumber \\
=\frac{z^{\frac{k}{2}-1}\pi^{\frac{k}{4}-\frac{1}{2}}x^{-\frac{k}{4}}}{2^{\frac{k}{2}-1}\Gamma\left(\frac{k}{2}\right)}\,e^{-z^{2}/8}+\frac{e^{-z^{2}/8}}{\sqrt{x}}\,\sum_{n=1}^{\infty}\frac{r_{k}(n)}{n^{\frac{k}{4}-\frac{1}{2}}}\,e^{-\frac{\pi n}{x}}\,I_{\frac{k}{2}-1}\left(\sqrt{\frac{\pi n}{x}}z\right),\label{Popov intro}
\end{align}
where $r_{k}(n)$ denotes the number of representations of the positive
integer $n$ as a sum of $k$ squares and, as usual, $J_{\nu}(z)$
and $I_{\nu}(z)$ respectively denote the Bessel and modified Bessel
functions of the first kind. 
A couple of reasons why this identity
is fascinating are already provided by Berndt, Dixit, Kim and Zaharescu
in {[}\cite{berndt_popov}, pp. 3795-3796{]}. For the purposes of our discussion, we
shall enumerate them.
\begin{enumerate}
\item If we construct the Dirichlet series attached to $r_{k}(n)$, 
\begin{equation}
\zeta_{k}(s)=\sum_{n=1}^{\infty}\frac{r_{k}(n)}{n^{s}},\,\,\,\,\text{Re}(s)>\frac{k}{2},\label{definition zeta k}
\end{equation}
then $\zeta_{k}(s)$ can be continued to the complex plane as a meromorphic
function possessing only a simple pole at $s=\frac{k}{2}$ with residue
$\pi^{k/2}/\Gamma(k/2)$. Moreover, it satisfies the functional equation
\begin{equation}
\pi^{-s}\Gamma\left(s\right)\zeta_{k}(s)=\pi^{s-\frac{k}{2}}\Gamma\left(\frac{k}{2}-s\right)\zeta_{k}\left(\frac{k}{2}-s\right).\label{functional equation zetak}
\end{equation}
Note that, when $k=1$, $r_{1}(n)=2$ if and only if $n$ is a perfect
square and zero otherwise. Therefore, (\ref{definition zeta k}) reduces
to 
\begin{equation}
\zeta_{1}(s):=\sum_{n=1}^{\infty}\frac{r_{1}(n)}{n^{s}}=2\sum_{n=1}^{\infty}\frac{1}{n^{2s}}=2\zeta(2s),\,\,\,\,\text{Re}(s)>\frac{1}{2}.\label{zeta 1(s) definition}
\end{equation}
Furthermore, (\ref{functional equation zetak}) with $k=1$ gives the functional equation
for Riemann's $\zeta-$function
\begin{equation}
\pi^{-s}\Gamma\left(s\right)\zeta(2s)=\pi^{s-\frac{1}{2}}\Gamma\left(\frac{1}{2}-s\right)\zeta\left(1-2s\right).\label{Functional equation Riemann Popov}
\end{equation}
The first point highlighted in \cite{berndt_popov} is that the powers of $n$
in the denominators of both sides of (\ref{Popov intro}) are remindful
of the functional equation (\ref{functional equation zetak}).
\item Riemann's second proof of the functional equation for $\zeta(s)$,
(\ref{Functional equation Riemann Popov}), employs the transformation
formula for Jacobi's $\theta-$function,
\begin{equation}
\theta(x):=\sum_{n\in\mathbb{Z}}e^{-\pi n^{2}x}=\frac{1}{\sqrt{x}}\sum_{n\in\mathbb{Z}}e^{-\frac{\pi n^{2}}{x}}:=\frac{1}{\sqrt{x}}\theta\left(\frac{1}{x}\right),\,\,\,\,\text{Re}(x)>0.\label{jacobi theta intro popov}
\end{equation}
The theta transformation formula associated to the Dirichlet series
$\zeta_{k}(s)$ can be obtained by taking the $k^{\text{th}}$ power
on both sides of (\ref{jacobi theta intro popov}). This results in
the transformation
\begin{equation}
\sum_{n=0}^{\infty}r_{k}(n)\,e^{-\pi nx}=x^{-\frac{k}{2}}\sum_{n=0}^{\infty}r_{k}(n)\,e^{-\frac{\pi n}{x}},\,\,\,\,\text{Re}(x)>0.\label{theta k squares formula}
\end{equation}
Of course, the exponential factors on both sides of (\ref{Popov intro})
remind us the theta transformation formula (\ref{theta k squares formula}).
In fact, (\ref{theta k squares formula}) is a particular case of
(\ref{Popov intro}) when we let $z\rightarrow0$, due to the limiting relations for the Bessel functions {[}\cite{NIST}, p. 223,
eq. (10.7.3){]}
\[
\lim_{y\rightarrow0}y^{-\nu}J_{\nu}(y)=\lim_{y\rightarrow0}y^{-\nu}I_{\nu}(y)=\frac{2^{-\nu}}{\Gamma(\nu+1)}.
\]
\item Chandrasekharan and Narasimhan {[}\cite{arithmetical identities}, p. 19, eq. (65){]} proved yet
another equivalent identity to (\ref{functional equation zetak})
and (\ref{theta k squares formula}). If $x>0$ and $q>\frac{k-1}{2}$,
then
\begin{equation}
\frac{1}{\Gamma(q+1)}\,\sum_{0\leq n\leq x}{}^{^{\prime}}r_{k}(n)\,(x-n)^{q}=\frac{\pi^{\frac{k}{2}}x^{\frac{k}{2}+q}}{\Gamma\left(q+1+\frac{k}{2}\right)}+\pi^{-q}\sum_{n=1}^{\infty}r_{k}(n)\left(\frac{x}{n}\right)^{\frac{k}{4}+\frac{q}{2}}J_{\frac{k}{2}+q}\left(2\pi\sqrt{nx}\right),\label{Bessel expansion riesz sums}
\end{equation}
where the Bessel series on the right-hand side converges absolutely.
The prime on the summation sign indicates that, if $q=0$ and $x$
is an integer, then the last contribution in this Riesz sum is just
$\frac{1}{2}r_{k}(x)$. The appearance of the Bessel functions in
(\ref{Popov intro}) reminds us of (\ref{Bessel expansion riesz sums}).
\end{enumerate}

Berndt, Dixit, Kim and Zaharescu emphasized the importance of Popov's
result by explaining its connection with the formulas stated in the previous three items. We
would like to add another item to the list, which is perhaps more
directly related to the second item above.  When $k=1$, (\ref{Popov intro}) gives the identity
\[
x^{\frac{1}{4}}\,e^{z^{2}/8}\left\{ 1+2\sum_{n=1}^{\infty}e^{-\pi n^{2}x}\,\cos\left(\sqrt{\pi x}\,nz\right)\right\} =x^{-\frac{1}{4}}\,e^{-z^{2}/8}\left\{ 1+2\sum_{n=1}^{\infty}e^{-\frac{\pi n^{2}}{x}}\,\cosh\left(\sqrt{\frac{\pi}{x}}\,nz\right)\right\}, 
\]
which is mainly due to the particular cases of the Bessel functions [\cite{temme_book}, p. 248],
\[
J_{-\frac{1}{2}}(x)=\sqrt{\frac{2}{\pi x}}\,\cos(x),\,\,\,\,I_{-\frac{1}{2}}(x)=\sqrt{\frac{2}{\pi x}}\,\cosh(x).
\]

Aiming to study the zeros of shifted combinations of Dirichlet series \cite{ryce_published}, Yakubovich and the author of this note have
discovered, independently of Popov, 
an extension of (\ref{Popov intro}) to a class of Hecke Dirichlet
series.\footnote{Such extension was stated for the first time by Berndt in \cite{dirichletserisV}.
Berndt used a generalized version of Vorono\"i's summation formula.} The main
idea in \cite{ryce_published} is that (\ref{Popov intro}) can be achieved through
an integral representation of the form\footnote{When $k=1$, (\ref{integral formula ralpha given intro}) gives a formula due to Dixit {[}\cite{Dixit_theta}, p. 374, eq. (1.13){]}.}
\begin{align}
&\sqrt{x}\,e^{z^{2}/8}\,\sum_{n=1}^{\infty}\frac{r_{k}(n)}{n^{\frac{k}{4}-\frac{1}{2}}}\,e^{-\pi nx}\,J_{\frac{k}{2}-1}(\sqrt{\pi nx}\,z)-\frac{\pi^{\frac{k}{4}-\frac{1}{2}}z^{\frac{k}{2}-1}e^{-z^{2}/8}}{2^{\frac{k}{2}-1}\Gamma\left(\frac{k}{2}\right)x^{\frac{k}{4}}}\nonumber \\
=&\frac{e^{-z^{2}/8}}{\sqrt{x}}\,\sum_{n=1}^{\infty}\frac{r_{k}(n)}{n^{\frac{k}{4}-\frac{1}{2}}}\,e^{-\frac{\pi n}{x}}\,I_{\frac{k}{2}-1}\left(\sqrt{\frac{\pi n}{x}}z\right)-\,\frac{\pi^{\frac{k}{4}-\frac{1}{2}}z^{\frac{k}{2}-1}x^{\frac{k}{4}}\,e^{z^{2}/8}}{2^{\frac{k}{2}-1}\Gamma\left(\frac{k}{2}\right)}\nonumber \\
=&\left(\frac{\sqrt{\pi}z}{2}\right)^{\frac{k}{2}-1}\frac{e^{z^{2}/8}}{2\pi\Gamma\left(\frac{k}{2}\right)}\,\intop_{-\infty}^{\infty}\pi^{-\frac{k}{4}-it}\Gamma\left(\frac{k}{4}+it\right)\zeta_{k}\left(\frac{k}{4}+it\right)\,_{1}F_{1}\left(\frac{k}{4}+it;\,\frac{k}{2};\,-\frac{z^{2}}{4}\right)\,x^{-it}\,dt.\label{integral formula ralpha given intro}
\end{align}

The main advantage of the representation (\ref{integral formula ralpha given intro}) is that Popov's formula can be now expressed as a corollary of the symmetries of an integral involving the important Dirichlet series $\zeta_{k}(s)$. 
When it comes to study analogues of a summation formula and the connection
between it and the behavior of the associated Dirichlet series,
it is very important to keep a certain symmetry in the analytic structure
of both sides of the identity, which is helpful to preserve the ``modular
shape'' of it. 

Looking at Popov's formula (\ref{Popov intro}), one may see that
the indices of the Bessel functions appearing at both sides of it depend on $k$. Therefore, in our search for a new analogue of (\ref{Popov intro}), it is not unreasonable to start with the study of a more general series of the form
\begin{equation}
\sum_{n=1}^{\infty}\frac{r_{k}(n)}{n^{\frac{\nu}{2}}}\,e^{-\pi nx}\,J_{\nu}(\sqrt{\pi nx}z),\label{series generalist}
\end{equation}
where $\nu$ is some complex parameter which is independent of $k$.
Getting a formula for the series (\ref{series generalist}) would
provide a generalization of Popov's formula and, hopefully, furnish
a beautiful analogue of it. The study of this series, however, ultimately leads to
a huge disappointment.

In fact, as it can be seen in \cite{ribeiro_product_bessel}, for any $\text{Re}(\nu)>-1$ and
$\text{Re}(x)>0,\,\,z\in\mathbb{C}$, the following summation formula
holds
\begin{align}
\frac{x^{\frac{\nu+1}{2}}z^{\nu}\pi^{\frac{\nu}{2}}}{2^{\nu}\Gamma(\nu+1)}e^{z^{2}/8}&+\sqrt{x}e^{z^{2}/8}\sum_{n=1}^{\infty}\frac{r_{k}(n)}{n^{\frac{\nu}{2}}}\,e^{-\pi xn}\,J_{\nu}(\sqrt{\pi nx}z) =\frac{\pi^{\frac{\nu}{2}}x^{\frac{\nu-k+1}{2}}z^{\nu}}{2^{\nu}\Gamma(\nu+1)}e^{z^{2}/8}\,{}_{1}F_{1}\left(\frac{k}{2};\,\nu+1;\,-\frac{z^{2}}{4}\right)\nonumber \\
&+\frac{x^{\frac{\nu-k+1}{2}}z^{\nu}\pi^{\frac{\nu}{2}}}{2^{\nu}\Gamma(\nu+1)}\,e^{-z^{2}/8}\,\sum_{n=1}^{\infty}r_{k}(n)e^{-\frac{\pi n}{x}} \,\Phi_{3}\left(1-\frac{k}{2}+\nu;\,\nu+1;\,\frac{z^{2}}{4},\,\frac{\pi z^{2}n}{4x}\right),\label{formula for Phi3}
\end{align}
where $\Phi_{3}(b;c;w,u)$ is the usual Humbert function,
\[
\Phi_{3}(b;c;\,w,u)=\sum_{k,m=0}^{\infty}\frac{(b)_{k}}{(c)_{k+m}}\frac{w^{k}u^{m}}{k!\,m!}.
\]

\bigskip{}

Why is (\ref{formula for Phi3}) not as beautiful as (\ref{Popov intro})?
One reason for this is perhaps the generality of (\ref{formula for Phi3}), as (\ref{formula for Phi3}) implies Popov's formula when $\nu=\frac{k}{2}-1$. The symmetries of (\ref{Popov intro}) are lost once we pick a general index for the Bessel function
$J_{\nu}(z)$.

Besides the loss of symmetry, another reason why (\ref{formula for Phi3})
is not appealing at all is that the analytical structures of both sides
of it are drastically different. We have started the summation formula at
the left with the Bessel function of the first kind, $J_{\nu}(z)$
(which, in hypergeometric terms, is nothing but a function of the
form $_{0}F_{1}$), but we ended up on the right with a confluent hypergeometric
function of two variables!

\bigskip{}

Although this first attempt fails in bringing an interesting analogue of (\ref{Popov intro}), we note that, if we slightly change the argument and indices of the Bessel functions in (\ref{Popov intro}), we can achieve a new analogue of Popov's result. This is stated as our main result.

\begin{theorem}\label{analogue theorem popov}
If $x>y>0$, then we
have the transformation formula
\begin{align}
\frac{(\pi y)^{\frac{k}{4}-\frac{1}{2}}}{2^{\frac{k}{4}-\frac{1}{2}}\Gamma\left(\frac{k}{4}+\frac{1}{2}\right)}&+\sum_{n=1}^{\infty}\frac{r_{k}(n)}{n^{\frac{k}{4}-\frac{1}{2}}}\,e^{-\pi nx}J_{\frac{k}{4}-\frac{1}{2}}\left(\pi ny\right)\nonumber \\
=\frac{(\pi y)^{\frac{k}{4}-\frac{1}{2}}}{2^{\frac{k}{4}-\frac{1}{2}}\Gamma\left(\frac{k}{4}+\frac{1}{2}\right)\left(x^{2}+y^{2}\right)^{\frac{k}{4}}}&+\frac{1}{\sqrt{x^{2}+y^{2}}}\,\sum_{n=1}^{\infty}\frac{r_{k}(n)}{n^{\frac{k}{4}-\frac{1}{2}}}\,e^{-\frac{\pi nx}{x^{2}+y^{2}}} J_{\frac{k}{4}-\frac{1}{2}}\left(\frac{\pi ny}{x^{2}+y^{2}}\right).\label{Formula to prove analogue Popov}
\end{align}
Moreover, under the same conditions, 
\begin{align}
\frac{(\pi y)^{\frac{k}{4}-\frac{1}{2}}}{2^{\frac{k}{4}-\frac{1}{2}}\Gamma\left(\frac{k}{4}+\frac{1}{2}\right)}&+\ensuremath{\sum_{n=1}^{\infty}\frac{r_{k}(n)}{n^{\frac{k}{4}-\frac{1}{2}}}\,e^{-\pi nx}I_{\frac{k}{4}-\frac{1}{2}}\left(\pi ny\right)}\nonumber \\
=\frac{(\pi y)^{\frac{k}{4}-\frac{1}{2}}}{2^{\frac{k}{4}-\frac{1}{2}}\Gamma\left(\frac{k}{4}+\frac{1}{2}\right)\left(x^{2}-y^{2}\right)^{\frac{k}{4}}}&+\frac{1}{\sqrt{x^{2}-y^{2}}}\sum_{n=1}^{\infty}\frac{r_{k}(n)}{n^{\frac{k}{4}-\frac{1}{2}}}\,e^{-\frac{\pi nx}{x^{2}-y^{2}}} I_{\frac{k}{4}-\frac{1}{2}}\left(\frac{\pi ny}{x^{2}-y^{2}}\right).\label{formula for modified bessel}
\end{align}
\end{theorem}

\bigskip{}

The reasons for the interest of (\ref{Formula to prove analogue Popov})
and (\ref{formula for modified bessel}) are almost the same as the
reasons why (\ref{Popov intro}) is fascinating. Due to the appearance
of the Bessel functions $J_{\nu}(z)$ and $I_{\nu}(z)$ and the exponential
factors, as well as the same powers on the denominators of both sides
of (\ref{Formula to prove analogue Popov}), our identity matches
the criteria mentioned in items 1-3. Moreover, if we define a generalized
$\theta-$function in the following way
\[
\Theta_{k}\left(x,y\right):=\frac{(\pi y)^{\frac{k}{4}-\frac{1}{2}}}{2^{\frac{k}{4}-\frac{1}{2}}\Gamma\left(\frac{k}{4}+\frac{1}{2}\right)}+\sum_{n=1}^{\infty}\frac{r_{k}(n)}{n^{\frac{k}{4}-\frac{1}{2}}}\,e^{-\pi nx}J_{\frac{k}{4}-\frac{1}{2}}\left(\pi ny\right),
\]
then (\ref{Formula to prove analogue Popov}) implies the transformation
formula
\begin{equation}
\Theta_{k}\left(x,y\right)=\frac{1}{\sqrt{x^{2}+y^{2}}}\Theta_{k}\left(\frac{x}{x^{2}+y^{2}},\frac{y}{x^{2}+y^{2}}\right), \label{general theta Popov Paper with Theta_k}
\end{equation}
which also reduces to (\ref{jacobi theta intro popov}) in the limit
$y\rightarrow0^{+}$.

\bigskip{}

This paper is organized as follows. In section \ref{preliminary results popov} we give some auxiliary
results to prove (\ref{Formula to prove analogue Popov}), which mainly concern the behavior of Gauss' hypergeometric function $_{2}F_{1}\left(a,b;c;z\right)$.
In section \ref{proof theorem popov section}, we prove rigorously (\ref{Formula to prove analogue Popov})
by using a variant of our approach to get (\ref{Popov intro}) given 
in \cite{ryce_published}. Lastly, the final section \ref{concluding remaks guinand popov paper} is devoted to a proof of a generalization of the Ramanujan-Guinand formula \cite{Guinand_Ramanujan}, which employs the same method as the proof of our Theorem \ref{analogue theorem popov}.

\bigskip{}

Before moving to the next section, let us remark that (\ref{Formula to prove analogue Popov}) and (\ref{formula for modified bessel}) contain some interesting analogues and particular cases. First, if we take $k=4$ in (\ref{Formula to prove analogue Popov}), we obtain the curious identity
\begin{equation*}
2\pi y+\sum_{n=1}^{\infty}\frac{r_{4}(n)}{n}\,\left\{ e^{-\pi n(x-y)}-e^{-\pi n(x+y)}\right\}=\frac{2\pi y}{x^{2}-y^{2}}+\sum_{n=1}^{\infty}\frac{r_{4}(n)}{n}\,\left\{ e^{-\frac{\pi n}{x+y}}-e^{-\frac{\pi n}{x-y}}\right\},
\end{equation*}
valid for $x>y>0$. 

It is also no surprise that we can extend our formulas
(\ref{Formula to prove analogue Popov}) and (\ref{formula for modified bessel}) to a more general class
of Dirichlet series satisfying Hecke's functional equation in the same spirit as \cite{arithmetical identities, berndt_class_identities, berndt_general_bessel} . We list some interesting examples.
\begin{enumerate}
\item Let $k\geq3$ be an odd positive integer and consider the divisor
function $\sigma_{k}(n)$. Then the Dirichlet series associated to
it, $\zeta(s)\zeta(s-k)$, will satisfy a functional equation similar
to (\ref{functional equation zetak}). Adapting the proof of Theorem \ref{analogue theorem popov} below, one can show the following identity
\begin{align*}
-\frac{B_{k+1}(\pi y)^{\frac{k}{2}}}{2(k+1)\Gamma\left(\frac{k}{2}+1\right)}&+\sum_{n=1}^{\infty}\frac{\sigma_{k}(n)}{n^{\frac{k}{2}}}\,e^{-2\pi nx}J_{\frac{k}{2}}\left(2\pi ny\right)\\
=\frac{(-1)^{\frac{k-1}{2}}B_{k+1}\,(\pi y)^{\frac{k}{2}}}{2(k+1)\Gamma\left(\frac{k}{2}+1\right)\left(x^{2}+y^{2}\right)^{\frac{k+1}{2}}}&+\frac{(-1)^{\frac{k+1}{2}}}{\sqrt{x^{2}+y^{2}}}\sum_{n=1}^{\infty}\frac{\sigma_{k}(n)}{n^{\frac{k}{2}}}e^{-\frac{2\pi nx}{x^{2}+y^{2}}}\,J_{\frac{k}{2}}\left(\frac{2\pi ny}{x^{2}+y^{2}}\right),
\end{align*}
which is valid for $x>y>0$.
\item For $z\in\mathbb{H}:=\left\{ z\in\mathbb{C}\,|\,\text{Im}(z)>0\right\} $, let $f(z)$
be a holomorphic cusp form with weight $k\geq12$ for the full modular
group with Fourier coefficients $a_{f}(n)$. The $L$-function associated
to $f(z)$, $L_{f}(s)$, satisfies an analogue of the functional equation
(\ref{functional equation zetak}). Using the same formalism as in this paper, we may get the formula
\[
\sum_{n=1}^{\infty}\frac{a_{f}(n)}{n^{\frac{k-1}{2}}}\,e^{-2\pi nx}J_{\frac{k-1}{2}}\left(2\pi ny\right)=\frac{(-1)^{k/2}}{\sqrt{x^{2}+y^{2}}}\sum_{n=1}^{\infty}\frac{a_{f}(n)}{n^{\frac{k-1}{2}}}\,e^{-\frac{2\pi nx}{x^{2}+y^{2}}}J_{\frac{k-1}{2}}\left(\frac{2\pi ny}{x^{2}+y^{2}}\right),\,\,\,\,x>y>0.
\]
In particular, the following particular case for Ramanujan's $\tau-$function takes place
\[
\sum_{n=1}^{\infty}\frac{\tau(n)}{n^{\frac{11}{2}}}\,e^{-2\pi nx}J_{\frac{11}{2}}\left(2\pi ny\right)=\frac{1}{\sqrt{x^{2}+y^{2}}}\,\sum_{n=1}^{\infty}\frac{\tau(n)}{n^{\frac{11}{2}}}\,e^{-\frac{2\pi nx}{x^{2}+y^{2}}}J_{\frac{11}{2}}\left(\frac{2\pi ny}{x^{2}+y^{2}}\right),\,\,\,\,x>y>0.
\]
\item Let $\chi$ be a nonprincipal, primitive and even Dirichlet character
modulo $q$. Its $L-$function, $L(2s,\chi)$, satisfies an analogue
of the functional equation (\ref{functional equation zetak}). Thus,
for $x>y>0$, the following formula holds
\[
\sum_{n=1}^{\infty}\chi(n)\sqrt{n}\,e^{-\frac{\pi n^{2}}{q}x}J_{-\frac{1}{4}}\left(\frac{\pi n^{2}y}{q}\right)=\frac{G(\chi)}{\sqrt{q\left(x^{2}+y^{2}\right)}}\,\sum_{n=1}^{\infty}\overline{\chi}(n)\sqrt{n}\,e^{-\frac{\pi n^{2}x}{q(x^{2}+y^{2})}}\,J_{-\frac{1}{4}}\left(\frac{\pi n^{2}y}{q\left(x^{2}+y^{2}\right)}\right),
\]
where $G(\chi)$ denotes the Gauss sum
\[
G(\chi):=\sum_{\ell=1}^{q-1}\chi(\ell)\,e^{\frac{2\pi i\ell}{q}}.
\]
Analogously, if $\chi$ is a nonprincipal, primitive and odd Dirichlet
character modulo $q$, then the associated $L-$function, $L(2s-1,\chi)$,
also satisfies a functional equation similar to (\ref{functional equation zetak}).
This gives the identity
\[
\sum_{n=1}^{\infty}\sqrt{n}\,\chi(n)\,\,e^{-\frac{\pi n^{2}x}{q}}J_{\frac{1}{4}}\left(\frac{\pi n^{2}y}{q}\right)=-\frac{iG(\chi)}{\sqrt{q\left(x^{2}+y^{2}\right)}}\,\sum_{n=1}^{\infty}\sqrt{n}\,\overline{\chi}(n)\,e^{-\frac{\pi n^{2}x}{q(x^{2}+y^{2})}}\,J_{\frac{1}{4}}\left(\frac{\pi n^{2}y}{q\left(x^{2}+y^{2}\right)}\right),
\]
which holds for $x>y>0$.
\end{enumerate}

\section{Preliminary results}\label{preliminary results popov}

For $|z|<1$, the hypergeometric function $_{2}F_{1}(a,b;c;z)$ is
defined by the Gauss series,
\begin{equation}
_{2}F_{1}\left(a,b;c;z\right)=\sum_{\ell=0}^{\infty}\frac{(a)_{\ell}(b)_{\ell}}{(c)_{\ell}\ell!}\,z^{\ell}=\frac{\Gamma(c)}{\Gamma(a)\Gamma(b)}\sum_{\ell=0}^{\infty}\frac{\Gamma(a+\ell)\Gamma(b+\ell)}{\Gamma(c+\ell)\ell!}\,z^{\ell},\label{intro def 2f1}
\end{equation}
and defined elsewhere by analytic continuation. For example, the analytic
continuation can be given via Slater's theorem \cite{Yakubovich_Index_Transforms}.
There is, however, another method of continuing the series (\ref{intro def 2f1}) via Euler's integral
representation: if $\text{Re}(c)>\text{Re}(b)>0$ and $|\arg(1-z)|<\pi$,
then $_{2}F_{1}(a,b;c;z)$ satisfies the formula
\begin{equation}
_{2}F_{1}\left(a,b;c;z\right)=\frac{\Gamma(c)}{\Gamma(b)\Gamma(c-b)}\intop_{0}^{1}t^{b-1}(1-t)^{c-b-1}(1-zt)^{-a}dt,\label{Euler integral}
\end{equation}
where we have chosen the principal
branch $(1-tz)^{-a}=e^{-a\log(1-tz)}$, with $\log(1-tz)$ being real
for $z\in[0,1]$.

There are several applications of Euler's formula. Using it, one is able to
derive the two most famous transformation formulas for hypergeometric
functions,
\begin{equation}
_{2}F_{1}\left(a,b;c;z\right)=(1-z)^{-a}\,_{2}F_{1}\left(a,c-b;c;\frac{z}{z-1}\right),\label{pfaff}
\end{equation}
\begin{equation}
_{2}F_{1}\left(a,b;c;z\right)=(1-z)^{c-a-b}\,_{2}F_{1}\left(c-a,c-b;c;z\right),\label{Euler formula}
\end{equation}
respectively known as Pfaff's and Euler's formulas\footnote{Alternatively, under the terminology of \cite{Yakubovich_Index_Transforms}, they are also
called Boltz formula and self-transformation formula.}.
It is also known that, for any $c\neq0,-1,-2,...$ and fixed $z$, the
function $_{2}F_{1}\left(a,b;c;z\right)$ is an entire function of
$a,\,b$ and $c$ {[}\cite{NIST}, p. 384, 15.2(ii){]}.

Just like (\ref{Euler integral}), there are other interesting integral representation of the hypergeometric
function. The most important formula in this paper is the Mellin transform {[}\cite{handbook_marichev},
p. 155, eq. (3.10.3){]},
\[
\intop_{0}^{\infty}x^{s-1}e^{-\alpha x}J_{\nu}(\beta x)\,dx=\frac{\beta^{\nu}}{2^{\nu}\alpha^{s+\nu}}\,\frac{\Gamma(s+\nu)}{\Gamma(\nu+1)}\,_{2}F_{1}\left(\frac{s+\nu}{2},\frac{s+\nu+1}{2};\nu+1;-\frac{\beta^{2}}{\alpha^{2}}\right),
\]
valid for $\text{Re}(s)>-\text{Re}(\nu)$ and $\text{Re}(\alpha)>|\text{Im}(\beta)|$.
Taking $\nu=\frac{k}{4}-\frac{1}{2}$ in the above representation and assuming that $\alpha>\beta>0$,
we see that the formula
\begin{equation}
\intop_{0}^{\infty}x^{s-1}e^{-\alpha x}J_{\frac{k}{4}-\frac{1}{2}}(\beta x)\,dx=\left(\frac{\beta}{2\alpha}\right)^{\frac{k}{4}-\frac{1}{2}}\alpha^{-s}\,\frac{\Gamma\left(s+\frac{k}{4}-\frac{1}{2}\right)}{\Gamma\left(\frac{k}{4}+\frac{1}{2}\right)}\,_{2}F_{1}\left(\frac{s}{2}+\frac{k}{8}-\frac{1}{4},\,\frac{s}{2}+\frac{k}{8}+\frac{1}{4};\,\frac{k}{4}+\frac{1}{2};-\frac{\beta^{2}}{\alpha^{2}}\right)\label{Intergal formula beginning 2F1}
\end{equation}
must hold for any $\text{Re}(s)>\frac{1}{2}-\frac{k}{4}.$ Since $k\geq1$
by hypothesis, (\ref{Intergal formula beginning 2F1})
must be always valid for $\text{Re}(s)>\frac{1}{4}$. This suggests the following 
inversion formula, which holds for $\sigma>\frac{1}{4}$ and $\alpha>\beta>0$,
\begin{equation}
e^{-\alpha x}J_{\frac{k}{4}-\frac{1}{2}}(\beta x)=\frac{(\beta/2\alpha)^{\frac{k}{4}-\frac{1}{2}}}{2\pi i\,\Gamma\left(\frac{k}{4}+\frac{1}{2}\right)}\intop_{\sigma-i\infty}^{\sigma+i\infty}\,\Gamma\left(s+\frac{k}{4}-\frac{1}{2}\right)\,_{2}F_{1}\left(\frac{4s+k-2}{8},\,\frac{4s+k+2}{8};\,\frac{k}{4}+\frac{1}{2};-\frac{\beta^{2}}{\alpha^{2}}\right)(\alpha x)^{-s}ds.\label{Inversion intro!}
\end{equation}

\bigskip{}

Since it will be instructive later, let us check
that Mellin's inversion theorem is valid for our pair of functions and prove (\ref{Inversion intro!}). Of course, for $\alpha>\beta>0$, the function $f(x)=e^{-\alpha x}J_{\nu}(\beta x)$ is continuous
and, for any $\sigma>-\nu$, $x^{\sigma-1}f(x)\in L_{1}(\mathbb{R}_{+})$. Hence, in order to check the conditions of Mellin's inversion theorem,
one just needs to see that
\begin{equation}
\Gamma\left(\sigma+\frac{k}{4}-\frac{1}{2}+it\right)\,_{2}F_{1}\left(\frac{4\sigma+k-2}{8}+\frac{it}{2},\,\frac{4\sigma+k+2}{8}+\frac{it}{2};\,\frac{k}{4}+\frac{1}{2};-\frac{\beta^{2}}{\alpha^{2}}\right)\in L_{1}\left(\sigma-i\infty,\sigma+i\infty\right).\label{condition mellin transform}
\end{equation}
A way to verify (\ref{condition mellin transform}) is by invoking an
asymptotic expansion due to Watson \cite{watson_asymptotic} (cf. {[}\cite{ERDELIY_TRANSCENDENTAL},
Vol 1, p. 77{]}). However, since Watson's analysis is much stronger
and difficult than what we need in this note, we shall proceed with
a simpler idea, following closely an analogous argument given in [\cite{Yakubovich_Index_Transforms},
p. 22]. 

\begin{lemma}
For $\alpha>\beta>0$ and $\sigma,\nu\in\mathbb{R}$ are such that $\nu>-1$ and $\sigma>-\nu$, then the following
asymptotic expansion takes place
\begin{align}
\,_{2}F_{1}\left(\frac{\sigma+\nu}{2}+\frac{it}{2},\frac{\sigma+\nu+1}{2}+\frac{it}{2};\nu+1;-\frac{\beta^{2}}{\alpha^{2}}\right)&\nonumber \\
=\Gamma(\nu+1)\,\left(\frac{\beta t}{2\alpha}\right)^{-\nu}I_{\nu}\left(\frac{\beta t}{\alpha}\right)\left\{ 1+O\left(\frac{1}{t}\right)\right\}& ,\,\,\,\,\,\,\,\,\,t\rightarrow\infty.\label{asymptotic first part}
\end{align}
\end{lemma}

\begin{proof}
First, let us note that we have the uniform estimate for $\beta/\alpha<x_{0}<1$,
\begin{align*}
\left|\,_{2}F_{1}\left(\frac{\sigma+\nu}{2}+\frac{it}{2},\frac{\sigma+\nu+1}{2}+\frac{it}{2};\nu+1;-\frac{\beta^{2}}{\alpha^{2}}\right)\right|&\\
\leq\frac{\sqrt{\pi}\,2^{1-\sigma-\nu}\Gamma(\nu+1)}{\left|\Gamma\left(\frac{\sigma+\nu}{2}+\frac{it}{2}\right)\Gamma\left(\frac{\sigma+\nu+1}{2}+\frac{it}{2}\right)\right|}\sum_{\ell=0}^{\infty}\frac{\Gamma(\sigma+\nu+2\ell)}{\Gamma(\nu+1+\ell)\,\ell!}\,\left(\frac{x_{0}}{2}\right)^{2\ell}&,
\end{align*}
with the latter series being convergent due to the ratio test. Hence,
\begin{equation*}
\,_{2}F_{1}\left(\frac{\sigma+\nu}{2}+\frac{it}{2},\frac{\sigma+\nu+1}{2}+\frac{it}{2};\nu+1;-\frac{\beta^{2}}{\alpha^{2}}\right)
=\lim_{N\rightarrow\infty}\sum_{\ell=0}^{N}\frac{(-1)^{\ell}\,\left(\sigma+\nu+it\right)_{2\ell}}{(\nu+1)_{\ell}\,\ell!}\left(\frac{\beta}{2\alpha}\right)^{2\ell}.
\end{equation*}
We shall consider the asymptotic expansion (with respect to the parameter $t$)
of the previous finite sum over $\ell$. From an exact version of
Stirling's formula, we know that
\begin{equation}
\,\left(\sigma+\nu+it\right)_{2\ell}=(-1)^{\ell}t^{2\ell}\left(1+O\left(\frac{1}{t}\right)\right),\,\,\,\,t\rightarrow\infty,\label{asymptotic pochammer}
\end{equation}
where the remainder is uniform with respect to $0\leq\ell\leq N$
{[}\cite{Yakubovich_Index_Transforms}, p. 20, eq. (1.145){]}. Hence, as $t\rightarrow\infty$,
\[
\sum_{\ell=0}^{N}\frac{(-1)^{\ell}\,\left(\sigma+\nu+it\right)_{2\ell}}{(\nu+1)_{\ell}\,\ell!}\left(\frac{\beta}{2\alpha}\right)^{2\ell}=\sum_{\ell=0}^{N}\frac{1}{(\nu+1)_{\ell}\,\ell!}\left(\frac{\beta t}{2\alpha}\right)^{2\ell}+O\left(\frac{1}{t}\sum_{\ell=0}^{N}\frac{1}{(\nu+1)_{\ell}\,\ell!}\left(\frac{\beta t}{2\alpha}\right)^{2\ell}\right).
\]
Clearly, when $N\rightarrow \infty$, the main term in the above expression can be identified with the modified Bessel function of the first kind, $I_{\nu}(x)$. Thus, as $t\rightarrow \infty$, we are able to get
\begin{align*}
\,_{2}F_{1}\left(\frac{\sigma+\nu}{2}+\frac{it}{2},\frac{\sigma+\nu+1}{2}+\frac{it}{2};\nu+1;-\frac{\beta^{2}}{\alpha^{2}}\right)
=\left\{ 1+O\left(\frac{1}{t}\right)\right\} \,\lim_{N\rightarrow\infty}\sum_{\ell=0}^{N}\frac{1}{(\nu+1)_{\ell}\,\ell!}\left(\frac{\beta t}{2\alpha}\right)^{2\ell}\\
=\Gamma(\nu+1)\,\left(\frac{\beta t}{2\alpha}\right)^{-\nu}I_{\nu}\left(\frac{\beta t}{\alpha}\right)\left\{ 1+O\left(\frac{1}{t}\right)\right\}.
\end{align*}

\end{proof}

Take $\nu:=\frac{k}{4}-\frac{1}{2}$ in the asymptotic expansion (\ref{asymptotic first part}).
Since $\nu>-\frac{1}{2}$, we may use the Poisson integral {[}\cite{NIST},
p. 252, eq. (10.32.2){]},
\[
\left(\frac{z}{2}\right)^{-\nu}I_{\nu}(z)=\frac{1}{\sqrt{\pi}\Gamma(\nu+\frac{1}{2})}\,\intop_{-1}^{1}(1-t^{2})^{\nu-\frac{1}{2}}e^{zt}dt,\,\,\,\,\nu>-\frac{1}{2},\,\,\,z\in\mathbb{C},
\]
to reach the elementary bound
\begin{equation}
\left|\left(\frac{z}{2}\right)^{-\nu}I_{\nu}(z)\right|\leq\frac{e^{|\text{Re}(z)|}}{\Gamma(\nu+1)}.\label{bound final Iv}
\end{equation}
Using Stirling's formula, (\ref{asymptotic first part}) and (\ref{bound final Iv})
with $\nu=\frac{k}{4}-\frac{1}{2}$, one can easily derive the bound
\begin{align}
\left|\Gamma\left(\sigma+\frac{k}{4}-\frac{1}{2}+it\right)\,_{2}F_{1}\left(\frac{4\sigma+k-2}{8}+\frac{it}{2},\,\frac{4\sigma+k+2}{8}+\frac{it}{2};\,\frac{k}{4}+\frac{1}{2};-\frac{\beta^{2}}{\alpha^{2}}\right)\right|\nonumber \\
=(2\pi)^{\frac{1}{2}}\Gamma\left(\frac{k}{4}+\frac{1}{2}\right)\,|t|^{\sigma+\frac{k}{4}-1}e^{-\frac{\pi}{2}|t|}\left|\left(\frac{\beta t}{2\alpha}\right)^{-\frac{k}{4}+\frac{1}{2}}I_{\frac{k}{4}-\frac{1}{2}}\left(\frac{\beta t}{\alpha}\right)\right|\left\{ 1+O\left(\frac{1}{|t|}\right)\right\} \nonumber \\
\leq(2\pi)^{\frac{1}{2}}|t|^{\sigma+\frac{k}{4}-1}\exp\left(-\left(\frac{\pi}{2}-\frac{\beta}{\alpha}\right)|t|\right)\left\{ 1+O\left(\frac{1}{|t|}\right)\right\} ,\,\,\,\,\,|t|\rightarrow\infty,\label{estimate gamma x 2f1}
\end{align}
which shows (\ref{condition mellin transform}), because $\beta/\alpha<1$ by hypothesis. This inequality will be very useful in the sequel.

\section{Proof of Theorem \ref{analogue theorem popov}}\label{proof theorem popov section}

It suffices to prove the formula (\ref{Formula to prove analogue Popov}),
because the proof of (\ref{formula for modified bessel}) is totally
analogous and the only difficulty required for this is just to check
that all the steps below are also valid then $y$ is replaced by $iy$
in (\ref{Formula to prove analogue Popov}). Also, due to the inequality (\ref{bound final Iv}), the condition
$x>y>0$ ensures that both series on (\ref{formula for modified bessel})
are absolutely convergent. 

We start by transforming the infinite series on the left-hand side
of (\ref{Formula to prove analogue Popov}). Using the integral representation
(\ref{Inversion intro!}) with $\alpha=\pi x$, $\beta=\pi y$ and
$x=n$ and choosing $\sigma>\max\left\{ \frac{k}{4}+\frac{1}{2},\frac{3k}{4}-\frac{1}{2}\right\} $
in (\ref{Inversion intro!}), we find that
\begin{align*}
\sum_{n=1}^{\infty}\frac{r_{k}(n)}{n^{\frac{k}{4}-\frac{1}{2}}}\,e^{-\pi nx}J_{\frac{k}{4}-\frac{1}{2}}\left(\pi ny\right)=\frac{y^{\frac{k}{4}-\frac{1}{2}}}{(2x)^{\frac{k}{4}-\frac{1}{2}}\Gamma\left(\frac{k}{4}+\frac{1}{2}\right)}\,\sum_{n=1}^{\infty}\frac{r_{k}(n)}{n^{\frac{k}{4}-\frac{1}{2}}}\\
\times\frac{1}{2\pi i}\,\intop_{\sigma-i\infty}^{\sigma+i\infty}\Gamma\left(s+\frac{k}{4}-\frac{1}{2}\right)\,_{2}F_{1}\left(\frac{4s+k-2}{8},\,\frac{4s+k+2}{8};\,\frac{k}{4}+\frac{1}{2};-\frac{y^{2}}{x^{2}}\right) & \left(\pi xn\right)^{-s}ds.
\end{align*}
By our choice of $\sigma>\frac{k}{4}+\frac{1}{2}$ and by the fact
that the Dirichlet series defining $\zeta_{k}(s)$, (\ref{definition zeta k}),
converges absolutely in the region $\text{Re}(s)>\frac{k}{2}$, the
interchange of the orders of integration and summation is possible
due to absolute convergence. Thus, if $\sigma>\frac{k}{4}+\frac{1}{2}$,
we have the integral representation
\begin{align}
\sum_{n=1}^{\infty}\frac{r_{k}(n)}{n^{\frac{k}{4}-\frac{1}{2}}}\,e^{-\pi nx}J_{\frac{k}{4}-\frac{1}{2}}\left(\pi ny\right) & =\frac{y^{\frac{k}{4}-\frac{1}{2}}}{(2x)^{\frac{k}{4}-\frac{1}{2}}\,\Gamma\left(\frac{k}{4}+\frac{1}{2}\right)}\nonumber \\
\times\,\frac{1}{2\pi i}\,\intop_{\sigma-i\infty}^{\sigma+i\infty}\Gamma\left(s+\frac{k}{4}-\frac{1}{2}\right)\,\zeta_{k}\left(s+\frac{k}{4}-\frac{1}{2}\right) & \,_{2}F_{1}\left(\frac{4s+k-2}{8},\,\frac{4s+k+2}{8};\,\frac{k}{4}+\frac{1}{2};-\frac{y^{2}}{x^{2}}\right)\,(\pi x)^{-s}ds.\label{integral representation first approach}
\end{align}

We shift the line of integration in (\ref{integral representation first approach})
from $\text{Re}(s)=\sigma$ to $\text{Re}(s)=\frac{k}{2}-\sigma$ and apply Cauchy's residue
theorem by performing an integration over the positively oriented rectangular
contour $\left[\sigma\pm iT,\,\frac{k}{2}-\sigma\pm iT\right]$, $T>0$.
By the Phragm\'en-Lindel\"of principle \cite{titchmarsh_theory_of_functions}, we know that,
for any $\delta>0$, $\zeta_{k}(\sigma+it)\ll_{\delta}|t|^{A(\sigma)+\delta}$,
where
\begin{equation}
A(\sigma)=\begin{cases}
0 & \sigma>\frac{k}{2}+\delta\\
\frac{k}{2}-\sigma & -\delta\leq\sigma\leq\frac{k}{2}+\delta\\
\frac{k}{2}-2\sigma & \sigma<-\delta
\end{cases}. \label{convex estimates cases Popov paper zetak}
\end{equation}
Thus, we can easily find the estimate for the contribution of the horizontal segments, 
\begin{align*}
\intop_{\sigma\pm iT}^{\frac{k}{2}-\sigma\pm iT}\left|\Gamma\left(s+\frac{k}{4}-\frac{1}{2}\right)\,\zeta_{k}\left(s+\frac{k}{4}-\frac{1}{2}\right)\,_{2}F_{1}\left(\frac{4s+k-2}{8},\,\frac{4s+k+2}{8};\,\frac{k}{4}+\frac{1}{2};-\frac{y^{2}}{x^{2}}\right)\,(\pi x)^{-s}\right||ds|\\
\ll_{\delta,\sigma,x}\,(2\pi)^{\frac{1}{2}}\,T^{\sigma+A\left(\sigma+\frac{k}{4}-\frac{1}{2}\right)+\frac{k}{4}-1}\exp\left[-\left(\frac{\pi}{2}-\frac{y}{x}\right)T\right],
\end{align*}
which tends to zero as $T\rightarrow\infty$. Therefore, by Cauchy's
theorem, 
\begin{align}
\frac{1}{2\pi i}\,\intop_{\sigma-i\infty}^{\sigma+i\infty}\Gamma\left(s+\frac{k}{4}-\frac{1}{2}\right)\,\zeta_{k}\left(s+\frac{k}{4}-\frac{1}{2}\right)\,_{2}F_{1}\left(\frac{4s+k-2}{8},\,\frac{4s+k+2}{8};\,\frac{k}{4}+\frac{1}{2};-\frac{y^{2}}{x^{2}}\right)\,(\pi x)^{-s}ds\nonumber \\
=\frac{1}{2\pi i}\,\intop_{\frac{k}{2}-\sigma-i\infty}^{\frac{k}{2}-\sigma+i\infty}\Gamma\left(s+\frac{k}{4}-\frac{1}{2}\right)\,\zeta_{k}\left(s+\frac{k}{4}-\frac{1}{2}\right)\,_{2}F_{1}\left(\frac{4s+k-2}{8},\,\frac{4s+k+2}{8};\,\frac{k}{4}+\frac{1}{2};-\frac{y^{2}}{x^{2}}\right)\,(\pi x)^{-s}ds\nonumber \\
+\sum_{\rho}\text{Res}_{s=\rho}\left\{ \Gamma\left(s+\frac{k}{4}-\frac{1}{2}\right)\,\zeta_{k}\left(s+\frac{k}{4}-\frac{1}{2}\right)\,_{2}F_{1}\left(\frac{4s+k-2}{8},\,\frac{4s+k+2}{8};\,\frac{k}{4}+\frac{1}{2};-\frac{y^{2}}{x^{2}}\right)\,(\pi x)^{-s}\right\} ,\label{residue after lalala}
\end{align}
where the last sum denotes the contribution from the residues of the
integrand inside the contour. The hypergeometric function appearing in the previous integral is an entire function of $s$,
so the possible poles come from the factors involving the Gamma function
and the Dirichlet series $\zeta_{k}\left(s+\frac{k}{4}-\frac{1}{2}\right)$.
Since $\zeta_{k}(s)$ has trivial zeros at the points $s\in\mathbb{Z}^{-}$,
the only possible poles of the product $\Gamma\left(s+\frac{k}{4}-\frac{1}{2}\right)\,\zeta_{k}\left(s+\frac{k}{4}-\frac{1}{2}\right)$ inside the rectangular contour 
are located at $\rho=\frac{k}{4}+\frac{1}{2}$ and $\rho=\frac{1}{2}-\frac{k}{4}$. Routine calculations show that\footnote{Recall that the residue of $\zeta_{k}(s)$ at $s=\frac{k}{2}$ is
$\pi^{k/2}/\Gamma(k/2)$ and $\zeta_{k}(0)=-1$.}
\begin{equation}
\text{Res}_{s=\frac{k}{4}+\frac{1}{2}}\left\{ \mathcal{J}_{x,y}(s)\right\} =\pi^{\frac{k}{4}-\frac{1}{2}}x^{-\frac{k}{4}-\frac{1}{2}}\,_{2}F_{1}\left(\frac{k}{4},\,\frac{k}{4}+\frac{1}{2};\,\frac{k}{4}+\frac{1}{2};-\frac{y^{2}}{x^{2}}\right)=\frac{(\pi x)^{\frac{k}{4}-\frac{1}{2}}}{\left(x^{2}+y^{2}\right)^{\frac{k}{4}}},\label{residue 1111}
\end{equation}
\begin{equation}
\text{Res}_{s=\frac{1}{2}-\frac{k}{4}}\left\{ \mathcal{J}_{x,y}(s)\right\} =-(\pi x)^{\frac{k}{4}-\frac{1}{2}}.\label{residue 222222}
\end{equation}
Using (\ref{residue 1111}), (\ref{residue 222222}) and finally rearranging
the terms in (\ref{residue after lalala}), we are able to obtain
the representation
\begin{align}
\frac{(\pi y)^{\frac{k}{4}-\frac{1}{2}}}{2^{\frac{k}{4}-\frac{1}{2}}\,\Gamma\left(\frac{k}{4}+\frac{1}{2}\right)}+\sum_{n=1}^{\infty}\frac{r_{k}(n)}{n^{\frac{k}{4}-\frac{1}{2}}}\,e^{-\pi nx}J_{\frac{k}{4}-\frac{1}{2}}\left(\pi ny\right) & =\frac{(\pi y)^{\frac{k}{4}-\frac{1}{2}}}{2^{\frac{k}{4}-\frac{1}{2}}\,\Gamma\left(\frac{k}{4}+\frac{1}{2}\right)\left(x^{2}+y^{2}\right)^{\frac{k}{4}}}+\frac{y^{\frac{k}{4}-\frac{1}{2}}}{(2x)^{\frac{k}{4}-\frac{1}{2}}\,\Gamma\left(\frac{k}{4}+\frac{1}{2}\right)}\nonumber \\
\times\frac{1}{2\pi i}\,\intop_{\frac{k}{2}-\sigma-i\infty}^{\frac{k}{2}-\sigma+i\infty}\Gamma\left(s+\frac{k}{4}-\frac{1}{2}\right)\zeta_{k}\left(s+\frac{k}{4}-\frac{1}{2}\right) & \,_{2}F_{1}\left(\frac{4s+k-2}{8},\,\frac{4s+k+2}{8};\,\frac{k}{4}+\frac{1}{2};-\frac{y^{2}}{x^{2}}\right)\,(\pi x)^{-s}ds.\label{almost at the end remaining integral}
\end{align}
Note that the left-hand side and the first term on the right-hand
side of (\ref{almost at the end remaining integral}) are already
the first three terms given in (\ref{Formula to prove analogue Popov}).
Thus, we just need to see that the integral in (\ref{almost at the end remaining integral})
reduces to the infinite series on the right of (\ref{Formula to prove analogue Popov}).
This will be the final technical procedure in our proof.

Invoking the functional equation for $\zeta_{k}(s)$, (\ref{functional equation zetak}),
together with Euler's transformation formula (\ref{Euler formula}),
one can check that the integral on the right-hand side of (\ref{almost at the end remaining integral})
is equal to
\begin{align*}
=&\frac{\sqrt{x^{2}+y^{2}}}{\pi x}\,\intop_{\frac{k}{2}-\sigma-i\infty}^{\frac{k}{2}-\sigma+i\infty}\Gamma\left(\frac{k}{4}+\frac{1}{2}-s\right)\zeta_{k}\left(\frac{k}{4}+\frac{1}{2}-s\right)\,_{2}F_{1}\left(\frac{k+6-4s}{8},\,\frac{k+2-4s}{8};\,\frac{k}{4}+\frac{1}{2};-\frac{y^{2}}{x^{2}}\right)\,\left(\frac{\pi x}{x^{2}+y^{2}}\right)^{s}\,ds\\
=&\frac{(\pi x)^{\frac{k}{2}-1}}{\left(x^{2}+y^{2}\right)^{\frac{k-1}{2}}}\,\intop_{\sigma-i\infty}^{\sigma+i\infty}\Gamma\left(s-\frac{k}{4}+\frac{1}{2}\right)\zeta_{k}\left(s-\frac{k}{4}+\frac{1}{2}\right)\,_{2}F_{1}\left(\frac{4s+6-k}{8},\,\frac{4s+2-k}{8};\,\frac{k}{4}+\frac{1}{2};-\frac{y^{2}}{x^{2}}\right)\,\left(\frac{\pi x}{x^{2}+y^{2}}\right)^{-s}\,ds.
\end{align*}

Since we have chosen $\sigma>\max\left\{ \frac{k}{4}+\frac{1}{2},\frac{3k}{4}-\frac{1}{2}\right\} $,
the Dirichlet series $\zeta_{k}\left(s-\frac{k}{4}+\frac{1}{2}\right)$
is absolutely convergent on the line $\text{Re}(s)=\sigma$. Hence, by absolute convergent once more, 
\begin{align}
&\frac{(\pi x)^{\frac{k}{2}-1}}{\left(x^{2}+y^{2}\right)^{\frac{k-1}{2}}}\,\intop_{\sigma-i\infty}^{\sigma+i\infty}\Gamma\left(s-\frac{k}{4}+\frac{1}{2}\right)\zeta_{k}\left(s-\frac{k}{4}+\frac{1}{2}\right)\,_{2}F_{1}\left(\frac{4s+6-k}{8},\,\frac{4s+2-k}{8};\,\frac{k}{4}+\frac{1}{2};-\frac{y^{2}}{x^{2}}\right)\,\left(\frac{\pi x}{x^{2}+y^{2}}\right)^{-s}\,ds\nonumber \\
&=\frac{(\pi x)^{\frac{k}{2}-1}}{\left(x^{2}+y^{2}\right)^{\frac{k-1}{2}}}\,\sum_{n=1}^{\infty}\frac{r_{k}(n)}{n^{\frac{1}{2}-\frac{k}{4}}}\,\intop_{\sigma-i\infty}^{\sigma+i\infty}\Gamma\left(s-\frac{k}{4}+\frac{1}{2}\right)\,_{2}F_{1}\left(\frac{4s+6-k}{8},\,\frac{4s+2-k}{8};\,\frac{k}{4}+\frac{1}{2};-\frac{y^{2}}{x^{2}}\right)\,\left(\frac{\pi xn}{x^{2}+y^{2}}\right)^{-s}\,ds\nonumber \\
&=\frac{1}{\sqrt{x^{2}+y^{2}}}\,\sum_{n=1}^{\infty}\frac{r_{k}(n)}{n^{\frac{k}{4}-\frac{1}{2}}}\,\intop_{\sigma+1-\frac{k}{2}-i\infty}^{\sigma+1-\frac{k}{2}+i\infty}\Gamma\left(z+\frac{k}{4}-\frac{1}{2}\right)\,_{2}F_{1}\left(\frac{4z+k+2}{8},\,\frac{4z+k-2}{8};\,\frac{k}{4}+\frac{1}{2};-\frac{y^{2}}{x^{2}}\right)\,\left(\frac{\pi xn}{x^{2}+y^{2}}\right)^{-z}\,dz.\label{NINETEEEEN!}
\end{align}
By hypothesis, $\sigma+1-\frac{k}{2}>\frac{1}{4}$,
so we can invoke representation (\ref{Inversion intro!})
in the last step with $\alpha=\frac{\pi x}{x^{2}+y^{2}}$, $\beta=\frac{\pi y}{x^{2}+y^{2}}$
and $x=n$, which yields
\begin{align}
\intop_{\sigma+1-\frac{k}{2}-i\infty}^{\sigma+1-\frac{k}{2}+i\infty}\Gamma\left(z+\frac{k}{4}-\frac{1}{2}\right)\,_{2}F_{1}\left(\frac{4z+k+2}{8},\,\frac{4z+k-2}{8};\,\frac{k}{4}+\frac{1}{2};-\frac{y^{2}}{x^{2}}\right)\,\left(\frac{\pi xn}{x^{2}+y^{2}}\right)^{-z}\,dz\nonumber \\
=\frac{2\pi i\,(2x)^{\frac{k}{4}-\frac{1}{2}}\Gamma\left(\frac{k}{4}+\frac{1}{2}\right)}{y^{\frac{k}{4}-\frac{1}{2}}}\,\exp\left(-\frac{\pi nx}{x^{2}+y^{2}}\right)\,J_{\frac{k}{4}-\frac{1}{2}}\left(\frac{\pi ny}{x^{2}+y^{2}}\right).\label{tweeenty}
\end{align}
Returning to (\ref{almost at the end remaining integral}) and (\ref{NINETEEEEN!}),
the use of the integral representation (\ref{tweeenty}) shows at
last that (\ref{almost at the end remaining integral}) results in the transformation formula
\begin{align*}
\frac{(\pi y)^{\frac{k}{4}-\frac{1}{2}}}{2^{\frac{k}{4}-\frac{1}{2}}\Gamma\left(\frac{k}{4}+\frac{1}{2}\right)}&+\sum_{n=1}^{\infty}\frac{r_{k}(n)}{n^{\frac{k}{4}-\frac{1}{2}}}\,e^{-\pi nx}J_{\frac{k}{4}-\frac{1}{2}}\left(\pi ny\right)\\
=\frac{(\pi y)^{\frac{k}{4}-\frac{1}{2}}}{2^{\frac{k}{4}-\frac{1}{2}}\Gamma\left(\frac{k}{4}+\frac{1}{2}\right)\left(x^{2}+y^{2}\right)^{\frac{k}{4}}}&+\frac{1}{\sqrt{x^{2}+y^{2}}}\,\sum_{n=1}^{\infty}\frac{r_{k}(n)}{n^{\frac{k}{4}-\frac{1}{2}}}\,e^{-\frac{\pi nx}{x^{2}+y^{2}}} J_{\frac{k}{4}-\frac{1}{2}}\left(\frac{\pi ny}{x^{2}+y^{2}}\right),
\end{align*}
which is precisely what we have claimed. $\blacksquare$

\bigskip{}

\section{Concluding Remarks: A generalization of the Ramanujan-Guinand formula}\label{concluding remaks guinand popov paper}

On page 253 of his Lost Notebook \cite{Ramanujan_lost, Guinand_Ramanujan, Ramanujan_Notebook_Lost}, Ramanujan
states the following formula, quoted from [\cite{Ramanujan_Notebook_Lost},
p. 94].

\paragraph*{Entry 3.3.1. (p. 253):}
Let $\sigma_{k}(n)=\sum_{d|n}d^{k}$ and let $K_{\nu}(z)$ be the  modified Bessel function of the second kind.
If $\alpha$ and $\beta$ are two positive numbers such that $\alpha\beta=\pi^{2}$
and if $\nu$ is any complex number, then
\begin{align}
\sqrt{\alpha}\sum_{n=1}^{\infty}\sigma_{-\nu}(n)\,n^{\nu/2}K_{\frac{\nu}{2}}(2n\alpha) & -\sqrt{\beta}\sum_{n=1}^{\infty}\sigma_{-\nu}(n)\,n^{\nu/2}K_{\frac{\nu}{2}}(2n\beta)\nonumber \\
=\frac{1}{4}\Gamma\left(-\frac{\nu}{2}\right)\zeta(-\nu)\left\{ \beta^{(1+\nu)/2}-\alpha^{(1+\nu)/2}\right\}  & +\frac{1}{4}\Gamma\left(\frac{\nu}{2}\right)\zeta(\nu)\left\{ \beta^{(1-\nu)/2}-\alpha^{(1-\nu)/2}\right\} .\label{Guinand given at concluding rpopov}
\end{align}

\bigskip{}

This formula was rediscovered by Guinand in 1955 \cite{guinand_rapidly_convergent}, who employed Watson's
formula \cite{watson_reciprocal} in order to derive it. Using an idea similar
to the one employed by Guinand (but working with a general analogue
of Watson's formula), we have proved in \cite{ribeiro_product_bessel} the following
generalization of (\ref{Guinand given at concluding rpopov}).

\begin{theorem}\label{Guinand formula RPOPOV paper}
Let $r_{k}(n)$ denote the number of ways in which $n$ can be written
as the sum of $k$ squares. If $x,y>0$ and if $\nu$ is any complex
number, then the following formula holds
\begin{align}
2\Gamma\left(\frac{k}{2}\right)(\pi y)^{1-\frac{k}{2}}\,\sum_{m,n=1}^{\infty}r_{k}(m)\,r_{k}(n)\,\left(\frac{m}{n}\right)^{\frac{\nu}{2}}\left(mn\right)^{\frac{1}{2}-\frac{k}{4}}\,J_{\frac{k}{2}-1}\left(2\pi\sqrt{m\,n}\,y\right)\,K_{\nu}\left(2\pi\sqrt{m\,n}\,x\right)\nonumber \\
-\frac{2\Gamma\left(\frac{k}{2}\right)(\pi y)^{1-\frac{k}{2}}}{x^{2}+y^{2}}\,\sum_{m,n=1}^{\infty}r_{k}(m)\,r_{k}(n)\,\left(\frac{m}{n}\right)^{\frac{\nu}{2}}\left(mn\right)^{\frac{1}{2}-\frac{k}{4}}\,J_{\frac{k}{2}-1}\left(\frac{2\pi\sqrt{m\,n}\,y}{x^{2}+y^{2}}\right)\,K_{\nu}\left(\frac{2\pi\sqrt{m\,n}\,x}{x^{2}+y^{2}}\right)\nonumber \\
=\,x^{-\nu}\eta_{k}(\nu)\left\{ \frac{1}{\left(x^{2}+y^{2}\right)^{\frac{k}{2}-\nu}}-1\right\} +x^{\nu}\eta_{k}(-\nu)\left\{ \frac{1}{\left(x^{2}+y^{2}\right)^{\frac{k}{2}+\nu}}-1\right\} ,\label{guinand ryce popov paper}
\end{align}
where $\eta_{k}(s)$ denotes the completed Dirichlet series,
\[
\eta_{k}(s)=\pi^{-s}\Gamma(s)\,\zeta_{k}(s).
\]
Moreover, if $\nu$ is any complex number and $x>y>0$, then the analogous
formula is valid
\begin{align}
2\Gamma\left(\frac{k}{2}\right)(\pi y)^{1-\frac{k}{2}}\,\sum_{m,n=1}^{\infty}r_{k}(m)\,r_{k}(n)\,\left(\frac{m}{n}\right)^{\frac{\nu}{2}}\left(mn\right)^{\frac{1}{2}-\frac{k}{4}}\,I_{\frac{k}{2}-1}\left(2\pi\sqrt{m\,n}\,y\right)\,K_{\nu}\left(2\pi\sqrt{m\,n}\,x\right)\nonumber \\
-\frac{2\Gamma\left(\frac{k}{2}\right)(\pi y)^{1-\frac{k}{2}}}{x^{2}-y^{2}}\,\sum_{m,n=1}^{\infty}r_{k}(m)\,r_{k}(n)\,\left(\frac{m}{n}\right)^{\frac{\nu}{2}}\left(mn\right)^{\frac{1}{2}-\frac{k}{4}}\,I_{\frac{k}{2}-1}\left(\frac{2\pi\sqrt{m\,n}\,y}{x^{2}-y^{2}}\right)\,K_{\nu}\left(\frac{2\pi\sqrt{m\,n}\,x}{x^{2}-y^{2}}\right)\nonumber \\
=\,x^{-\nu}\eta_{k}(\nu)\left\{ \frac{1}{\left(x^{2}-y^{2}\right)^{\frac{k}{2}-\nu}}-1\right\} +x^{\nu}\eta_{k}(-\nu)\left\{ \frac{1}{\left(x^{2}-y^{2}\right)^{\frac{k}{2}+\nu}}-1\right\} .\label{guinand analogue I modified}
\end{align}
\end{theorem}

\bigskip{}

At first glance, (\ref{guinand ryce popov paper}) and (\ref{guinand analogue I modified})
do not seem to be related to (\ref{Guinand given at concluding rpopov}),
despite the fact that both formulas share the Modified Bessel function,
$K_{\nu}(z)$, as an element in their infinite series. In the next
few lines, we shall argue why (\ref{guinand ryce popov paper}) and
(\ref{guinand analogue I modified}) are, in fact, generalizations
of (\ref{Guinand given at concluding rpopov}). Recalling that $\zeta_{1}(s)=2\zeta(2s)$
(see (\ref{zeta 1(s) definition}) above) and that $r_{1}(n)=2$ iff
$n$ is a perfect square and $0$ otherwise, we see that a particular
case of (\ref{guinand ryce popov paper}) is
\begin{align}
&8\pi\sqrt{y}\,\sum_{m,n=1}^{\infty}\left(\frac{m}{n}\right)^{\nu}\sqrt{m\,n}\,J_{-\frac{1}{2}}\left(2\pi m\,n\,y\right)\,K_{\nu}\left(2\pi m\,n\,x\right)\nonumber \\
-&\frac{8\pi\sqrt{y}}{x^{2}+y^{2}}\,\sum_{m,n=1}^{\infty}\left(\frac{m}{n}\right)^{\nu}\sqrt{m\,n}\,J_{-\frac{1}{2}}\left(\frac{2\pi m\,n\,y}{x^{2}+y^{2}}\right)\,K_{\nu}\left(\frac{2\pi m\,n\,x}{x^{2}+y^{2}}\right)\nonumber \\
&=2(\pi x)^{-\nu}\,\Gamma(\nu)\zeta\left(2\nu\right)\left\{ \frac{1}{\left(x^{2}+y^{2}\right)^{\frac{1}{2}-\nu}}-1\right\} +2(\pi x)^{\nu}\,\Gamma(-\nu)\zeta\left(-2\nu\right)\left\{ \frac{1}{\left(x^{2}+y^{2}\right)^{\frac{1}{2}+\nu}}-1\right\} .\label{immediate case k=00003D1}
\end{align}
However, this expression can be simplified further. Indeed, recalling
the reduction formula for the Bessel function [\cite{temme_book}, p. 248]
\[
J_{-\frac{1}{2}}(x)=\sqrt{\frac{2}{\pi x}}\,\cos(x),
\]
and making the substitution of $\nu$ by $\nu/2$ in (\ref{immediate case k=00003D1}), we
can now rewrite (\ref{immediate case k=00003D1}) in the following manner
\begin{align}
\sum_{m,n=1}^{\infty}\left(\frac{m}{n}\right)^{\nu/2}\cos\left(2\pi m\,n\,y\right)\,K_{\nu/2}\left(2\pi m\,n\,x\right)&-\frac{1}{\sqrt{x^{2}+y^{2}}}\,\sum_{m,n=1}^{\infty}\left(\frac{m}{n}\right)^{\nu/2}\,\cos\left(\frac{2\pi m\,n\,y}{x^{2}+y^{2}}\right)\,K_{\nu/2}\left(\frac{2\pi m\,n\,x}{x^{2}+y^{2}}\right)\nonumber \\
=\frac{1}{4}(\pi x)^{-\frac{\nu}{2}}\,\Gamma\left(\frac{\nu}{2}\right)\,\zeta\left(\nu\right)\left\{ \frac{1}{\left(x^{2}+y^{2}\right)^{\frac{1-\nu}{2}}}-1\right\} &+\frac{1}{4}(\pi x)^{\frac{\nu}{2}}\,\Gamma\left(-\frac{\nu}{2}\right)\zeta\left(-\nu\right)\left\{ \frac{1}{\left(x^{2}+y^{2}\right)^{\frac{1+\nu}{2}}}-1\right\} .\label{intermediate guuuinand}
\end{align}
The main formal difference between (\ref{Guinand given at concluding rpopov})
and the newly derived formula (\ref{intermediate guuuinand}) is that
(\ref{Guinand given at concluding rpopov}) is expressed via two single
infinite series, whilst the infinite series in (\ref{intermediate guuuinand})
contain two variables of summation. A way to circumvent this is
by recalling the definition of the divisor function
\[
\sigma_{z}(n)=\sum_{d|n}d^{z},
\]
which helps us to see that (\ref{intermediate guuuinand}) can be rewritten as 
\begin{align}
\,\sum_{n=1}^{\infty}\sigma_{-\nu}(n)\,n^{\nu/2}\cos\left(2\pi ny\right)\,K_{\nu/2}\left(2\pi n\,x\right)-\frac{1}{\sqrt{x^{2}+y^{2}}}\,\sum_{n=1}^{\infty}\,\sigma_{-\nu}(n)\,n^{\nu/2}\,\cos\left(\frac{2\pi n\,y}{x^{2}+y^{2}}\right)\,K_{\nu/2}\left(\frac{2\pi n\,x}{x^{2}+y^{2}}\right)\nonumber \\
=\frac{1}{4}(\pi x)^{-\frac{\nu}{2}}\,\Gamma\left(\frac{\nu}{2}\right)\zeta\left(\nu\right)\left\{ \frac{1}{\left(x^{2}+y^{2}\right)^{\frac{1-\nu}{2}}}-1\right\} +\frac{1}{4}(\pi x)^{\frac{\nu}{2}}\Gamma\left(-\frac{\nu}{2}\right)\,\zeta(-\nu)\left\{ \frac{1}{\left(x^{2}+y^{2}\right)^{\frac{1+\nu}{2}}}-1\right\} .\label{transformation formula Guinand popov final}
\end{align}
Let us note now that (\ref{transformation formula Guinand popov final}) is very similar to (\ref{Guinand given at concluding rpopov}), being even more general than it. If we replace $x$ by $\alpha/\pi$
and let $y=0$ in (\ref{transformation formula Guinand popov final}),
we can recover the Ramanujan-Guinand formula in the same form as (\ref{Guinand given at concluding rpopov}).

\bigskip{}

Having argued that (\ref{guinand ryce popov paper}) (respectively
(\ref{guinand analogue I modified})) constitutes a generalization
of the beautiful result of Ramanujan, why do we refer to it in this
note about Popov's formula? First, as we have shown in \cite{ribeiro_product_bessel},
(\ref{guinand ryce popov paper}) is an indirect consequence of (\ref{Popov intro}).
Second, the transformation formulas (\ref{guinand ryce popov paper})
and (\ref{guinand analogue I modified}) are clearly remindful of
(\ref{Formula to prove analogue Popov}) and (\ref{formula for modified bessel}),
because the roles of $x$ and $y$ in these formulas are transformed
in similar ways. In fact, for $\nu \in \mathbb{C}$, $k\in \mathbb{N}$ and $x,y>0$, if we set the function
\begin{align*}
\Psi_{k}\left(\nu;\,x,y\right)&:=x^{-\nu}\eta_{k}(\nu)+x^{\nu}\eta_{k}(-\nu)\\
+2\Gamma\left(\frac{k}{2}\right)(\pi y)^{1-\frac{k}{2}}\,&\sum_{m,n=1}^{\infty}r_{k}(m)\,r_{k}(n)\,\left(\frac{m}{n}\right)^{\frac{\nu}{2}}\left(m\,n\right)^{\frac{1}{2}-\frac{k}{4}}\,J_{\frac{k}{2}-1}\left(2\pi\sqrt{m\,n}\,y\right)\,K_{\nu}\left(2\pi\sqrt{m\,n}\,x\right),
\end{align*}
then (\ref{guinand ryce popov paper}) is equivalent to the transformation
\[
\Psi_{k}\left(\nu;\,x,y\right)=\frac{1}{\left(x^{2}+y^{2}\right)^{\frac{k}{2}}}\,\Psi_{k}\left(\nu;\,\frac{x}{x^{2}+y^{2}},\frac{y}{x^{2}+y^{2}}\right),
\]
which is very similar to (\ref{general theta Popov Paper with Theta_k}), provided at the introduction of this paper. Besides, the
proof of (\ref{guinand ryce popov paper}) (resp. (\ref{guinand analogue I modified})) here presented uses exactly the same argument as the proof of (\ref{Formula to prove analogue Popov}) (resp. (\ref{formula for modified bessel})).
Thus, in order to conclude our paper, we present a short proof of
the formulas (\ref{guinand ryce popov paper}) and (\ref{guinand analogue I modified}).

Just like before, we
start with a Mellin representation {[}\cite{handbook_marichev}, p. 224, eq. (3.14.12.1){]}
\begin{equation}
\intop_{0}^{\infty}x^{s-1}J_{\mu}\left(\beta x\right)\,K_{\nu}\left(\alpha x\right)dx=\frac{2^{s-2}\beta^{\mu}}{\alpha^{s+\mu}\Gamma(\mu+1)}\Gamma\left(\frac{s+\mu-\nu}{2}\right)\Gamma\left(\frac{s+\mu+\nu}{2}\right)\,_{2}F_{1}\left(\frac{s+\mu-\nu}{2},\frac{s+\mu+\nu}{2};\mu+1;-\frac{\beta^{2}}{\alpha^{2}}\right),\label{Direct mellin corollary ryce}
\end{equation}
which is valid for $\text{Re}(s+\mu)>|\text{Re}(\nu)|$ and $\text{Re}(\alpha)>|\text{Im}(\beta)|$.
Taking $\mu=\frac{k}{2}-1>-1$ and assuming that $\alpha>\beta>0$,
we find the Mellin inverse of (\ref{Direct mellin corollary ryce}),
\begin{align}
J_{\frac{k}{2}-1}\left(\beta x\right)\,K_{\nu}\left(\alpha x\right) & =\frac{1}{8\pi i\Gamma\left(\frac{k}{2}\right)}\left(\frac{\beta}{\alpha}\right)^{\frac{k}{2}-1}\,\intop_{\sigma-i\infty}^{\sigma+i\infty}\Gamma\left(\frac{s+k/2-1-\nu}{2}\right)\Gamma\left(\frac{s+k/2-1+\nu}{2}\right)\nonumber \\
 & \,\,\negmedspace\,\,\;\,\,\qquad\,\times\,_{2}F_{1}\left(\frac{s+k/2-1-\nu}{2},\frac{s+k/2-1+\nu}{2};\frac{k}{2};-\frac{\beta^{2}}{\alpha^{2}}\right)\left(\frac{\alpha x}{2}\right)^{-s}ds,\label{mellin inverse 2f1 ryce}
\end{align}
where $\sigma:=\text{Re}(s)>|\text{Re}(\nu)|+1-\frac{k}{2}$.

The conditions for the application of the Mellin inversion formula
(\ref{mellin inverse 2f1 ryce}) are met once we prove
an estimate analogous to (\ref{asymptotic first part}). This is done
in the next lemma.

\begin{lemma}\label{lemma concluding remark popov}
Let $\mu>-1$ and $\sigma+\mu>|\text{Re}(\nu)|$. Also, suppose that
$x>y>0$. Then there exists some sufficiently large $\tau_{0}$ such
that, for any $t\geq\tau_{0}$, the following inequality holds
\begin{align}
\left|\frac{\Gamma\left(\frac{\sigma+\mu-\nu}{2}+\frac{it}{2}\right)\Gamma\left(\frac{\sigma+\mu+\nu}{2}+\frac{it}{2}\right)}{\Gamma(\mu+1)}\,{}_{2}F_{1}\left(\frac{\sigma+\mu-\nu}{2}+\frac{it}{2},\frac{\sigma+\mu+\nu}{2}+\frac{it}{2};\,\mu+1;\,-\frac{\beta^{2}}{\alpha^{2}}\right)\right|\nonumber \\
\leq4\pi\,\left(\frac{t}{2}\right)^{\sigma+\mu-1}\,\exp\left\{ -\left(\frac{\pi}{2}-\frac{\beta}{\alpha}\right)t\right\} .\label{Inequality to prove corollary 2.11}
\end{align}
\end{lemma}

\begin{proof}
For $0<\beta/\alpha\leq X_{0}<1$, the use of the hypothesis $\sigma+\mu>|\text{Re}(\nu)|$
shows the uniform bound
\begin{align*}
\left|\frac{\Gamma\left(\frac{\sigma+\mu-\nu}{2}+\frac{it}{2}\right)\Gamma\left(\frac{\sigma+\mu+\nu}{2}+\frac{it}{2}\right)}{\Gamma(\mu+1)}\,{}_{2}F_{1}\left(\frac{\sigma+\mu-\nu}{2}+\frac{it}{2},\frac{\sigma+\mu+\nu}{2}+\frac{it}{2};\,\mu+1;\,-\frac{\beta^{2}}{\alpha^{2}}\right)\right|\\
\leq\sum_{\ell=0}^{\infty}\frac{\Gamma\left(\frac{\sigma+\mu-\text{Re}(\nu)}{2}+\ell\right)\Gamma\left(\frac{\sigma+\mu+\text{Re}(\nu)}{2}+\ell\right)}{\Gamma(\mu+1+\ell)\,\ell!}X_{0}^{2\ell},
\end{align*}
where the latter series is convergent due to the ratio test. Thus,
we have that
\begin{align*}
\frac{\Gamma\left(\frac{\sigma+\mu-\nu}{2}+\frac{it}{2}\right)\Gamma\left(\frac{\sigma+\mu+\nu}{2}+\frac{it}{2}\right)}{\Gamma(\mu+1)}\,{}_{2}F_{1}\left(\frac{\sigma+\mu-\nu}{2}+\frac{it}{2},\frac{\sigma+\mu+\nu}{2}+\frac{it}{2};\,\mu+1;\,-\frac{\beta^{2}}{\alpha^{2}}\right)\\
=\frac{\Gamma\left(\frac{\sigma+\mu-\nu}{2}+\frac{it}{2}\right)\Gamma\left(\frac{\sigma+\mu+\nu}{2}+\frac{it}{2}\right)}{\Gamma(\mu+1)}\,\lim_{N\rightarrow\infty}\sum_{\ell=0}^{N}\frac{(-1)^{\ell}\left(\frac{\sigma+\mu-\nu}{2}+\frac{it}{2}\right)_{\ell}\left(\frac{\sigma+\mu+\nu}{2}+\frac{it}{2}\right)_{\ell}}{(\mu+1)_{\ell}\ell!}\left(\frac{\beta}{\alpha}\right)^{2\ell}.
\end{align*}
Using the fact that the remainder in Stirling's formula can be uniformly
estimated with respect to the index $\ell$ (cf. [\cite{Yakubovich_Index_Transforms}, p. 20, eq. (1.145)]), we see that
\begin{align}
\frac{\Gamma\left(\frac{\sigma+\mu-\nu}{2}+\frac{it}{2}\right)\Gamma\left(\frac{\sigma+\mu+\nu}{2}+\frac{it}{2}\right)}{\Gamma(\mu+1)}\,\sum_{\ell=0}^{N}\frac{(-1)^{\ell}\left(\frac{\sigma+\mu-\nu}{2}+\frac{it}{2}\right)_{\ell}\left(\frac{\sigma+\mu+\nu}{2}+\frac{it}{2}\right)_{\ell}}{(\mu+1)_{\ell}\ell!}\left(\frac{\beta}{\alpha}\right)^{2\ell}\nonumber \\
=\frac{\Gamma\left(\frac{\sigma+\mu-\nu}{2}+\frac{it}{2}\right)\Gamma\left(\frac{\sigma+\mu+\nu}{2}+\frac{it}{2}\right)}{\Gamma(\mu+1)}\,\sum_{\ell=0}^{N}\frac{1}{(\mu+1)_{\ell}\ell!}\left(\frac{\beta t}{2\alpha}\right)^{2\ell}\left\{ 1+O\left(\frac{1}{t}\right)\right\} \nonumber \\
=\frac{2\pi}{\Gamma(\mu+1)}\left(\frac{t}{2}\right)^{\sigma+\mu-1}e^{-\frac{\pi}{2}t}\,\sum_{\ell=0}^{N}\frac{1}{(\mu+1)_{\ell}\ell!}\left(\frac{\beta t}{2\alpha}\right)^{2\ell}\left\{ 1+O\left(\frac{1}{t}\right)\right\} ,\label{After application of Stirling in concluding lemma}
\end{align}
where the last step is just an application of Stirling's formula for
the product $\Gamma\left(\frac{\sigma+\mu-\nu}{2}+\frac{it}{2}\right)\Gamma\left(\frac{\sigma+\mu+\nu}{2}+\frac{it}{2}\right)$.
Recalling the power series for the modified Bessel function $I_{\nu}(z)$,
(\ref{After application of Stirling in concluding lemma}) shows that
\begin{align*}
\frac{\Gamma\left(\frac{\sigma+\mu-\nu}{2}+\frac{it}{2}\right)\Gamma\left(\frac{\sigma+\mu+\nu}{2}+\frac{it}{2}\right)}{\Gamma(\mu+1)}\,{}_{2}F_{1}\left(\frac{\sigma+\mu-\nu}{2}+\frac{it}{2},\frac{\sigma+\mu+\nu}{2}+\frac{it}{2};\,\mu+1;\,-\frac{y^{2}}{x^{2}}\right)\\
=2\pi\left(\frac{t}{2}\right)^{\sigma+\mu-1}\left(\frac{\beta t}{2\alpha}\right)^{-\mu}I_{\mu}\left(\frac{\beta t}{\alpha}\right)\,e^{-\frac{\pi}{2}t}\left\{ 1+O\left(\frac{1}{t}\right)\right\} ,\,\,\,\,t\rightarrow\infty.
\end{align*}
Finally, appealing to (\ref{bound final Iv}) we find the desired inequality
(\ref{Inequality to prove corollary 2.11}), which proves Lemma \ref{lemma concluding remark popov}.
\end{proof}

By the previous lemma, one concludes that
\[
\frac{\Gamma\left(\frac{\sigma+\mu-\nu}{2}+\frac{it}{2}\right)\Gamma\left(\frac{\sigma+\mu+\nu}{2}+\frac{it}{2}\right)}{\Gamma(\mu+1)}\,{}_{2}F_{1}\left(\frac{\sigma+\mu-\nu}{2}+\frac{it}{2},\frac{\sigma+\mu+\nu}{2}+\frac{it}{2};\,\mu+1;\,-\frac{y^{2}}{x^{2}}\right)\in L_{1}\left(\sigma-i\infty,\sigma+i\infty\right),
\]
whenever $\sigma+\mu>|\text{Re}(\nu)|,$ $\mu>-1$ and $\beta>\alpha>0$.
Therefore, the integral representation (\ref{mellin inverse 2f1 ryce})
must hold under analogous conditions. We are now ready to give a new proof of Theorem \ref{Guinand formula RPOPOV paper} using the main ideas of this paper.
\begin{proof}[Proof of Theorem \ref{Guinand formula RPOPOV paper}]
As in the proof of Theorem \ref{analogue theorem popov}, it suffices for us to show the first formula
(\ref{guinand ryce popov paper}), as the arguments here presented
can be easily modified to give (\ref{guinand analogue I modified}).
Throughout this argument, we shall assume that $\text{Re}(\nu)>0$
and, for simplicity, that $\nu\neq\frac{k}{2}$. It is enough to prove
(\ref{guinand ryce popov paper}) under these simpler conditions because (\ref{guinand ryce popov paper})
is invariant under the reflection $\nu\leftrightarrow-\nu$ and the
extension to the point $\nu=\frac{k}{2}$ can be easily provided by
analytic continuation. Also, for simplicity in our notation, let
\begin{equation}
\eta_{k}(s):=\pi^{-s}\Gamma(s)\,\zeta_{k}\left(s\right).\label{definition once more etaK(s)}
\end{equation}

In this proof, we shall assume that $\sigma>\max\left\{ \frac{k}{2}+1+\text{Re}(\nu),k-1+\text{Re}(\nu)\right\} $.
By absolute convergence of the Dirichlet series defining $\zeta_{k}(s)$
in the region $\text{Re}(s)>\frac{k}{2}$, together with the Mellin
representation (\ref{mellin inverse 2f1 ryce}), we have 
\begin{align}
2\Gamma\left(\frac{k}{2}\right)(\pi y)^{1-\frac{k}{2}}\,\sum_{m,n=1}^{\infty}r_{k}(m)\,r_{k}(n)\,\left(\frac{m}{n}\right)^{-\frac{\nu}{2}}\left(m\,n\right)^{\frac{1}{2}-\frac{k}{4}}J_{\frac{k}{2}-1}\left(2\pi\sqrt{mn}\,y\right)\,K_{\nu}\left(2\pi\sqrt{mn}\,x\right)\nonumber \\
=\frac{x^{1-\frac{k}{2}}}{4\pi i}\,\intop_{\sigma-i\infty}^{\sigma+i\infty}\eta_{k}\left(\frac{s-\nu-1}{2}+\frac{k}{4}\right)\,\eta_{k}\left(\frac{s+\nu-1}{2}+\frac{k}{4}\right)\nonumber \\
\times\,_{2}F_{1}\left(\frac{s-\nu-1}{2}+\frac{k}{4},\frac{s+\nu-1}{2}+\frac{k}{4};\,\frac{k}{2};-\frac{y^{2}}{x^{2}}\right)x^{-s}\,ds.\label{starting point of proof of guinand ryce rppopov}
\end{align}

Considering the integral in (\ref{starting point of proof of guinand ryce rppopov}),
we shift its line of integration from $\text{Re}(s)=\sigma$ to $\text{Re}(s)=\frac{k}{2}-\sigma$
by integrating along a positively oriented rectangular contour with
vertices $\sigma\pm iT$ and $\frac{k}{2}-\sigma\pm iT$. Using Lemma \ref{lemma concluding remark popov} and the convex estimates for $\zeta_{k}(s)$, (\ref{convex estimates cases Popov paper zetak}), the integrals
along the horizontal segments, $\left[\frac{k}{2}-\sigma\pm iT,\,\sigma\pm iT\right]$
tend to zero as $T\rightarrow\infty$.

Since $\nu\neq\frac{k}{2}$ and $\text{Re}(\nu)>0$ by hypothesis,
we find that the integrand in (\ref{starting point of proof of guinand ryce rppopov})
contains four simple poles inside the rectangular contour $\left[\sigma\pm iT,\frac{k}{2}-\sigma\pm iT\right]$.
In fact, these poles are located at the points
\[
\rho_{1}=1+\nu-\frac{k}{2},\,\rho_{2}=\frac{k}{2}+1+\nu,\,\rho_{3}=1-\frac{k}{2}-\nu,\,\rho_{4}=\frac{k}{2}+1-\nu,
\]
and their simplicity is due to the conditions $\nu\neq\frac{k}{2}$ and $\text{Re}(\nu)>0$.
Let $\mathscr{R}_{j}(x)$, $j=1,...,4$, denote the respective residues of $\rho_{j}$, $j=1,...,4$:
from simple calculations, we have that 
\[
\mathscr{R}_{1}(x)=-2\eta_{k}\left(\nu\right)x^{\frac{k}{2}-\nu-1},
\]
\begin{align*}
\mathscr{R}_{2}(x) & =2\eta_{k}\left(\frac{k}{2}+\nu\right)\left(\frac{x}{2}\right)^{-\frac{k}{2}-1-\nu}\,_{2}F_{1}\left(\frac{k}{2},\frac{k}{2}+\nu;\frac{k}{2};-\frac{y^{2}}{x^{2}}\right)=2\eta_{k}\left(\frac{k}{2}+\nu\right)\,x^{\frac{k}{2}+\nu-1}\,\left(x^{2}+y^{2}\right)^{-\frac{k}{2}-\nu}\\
 & =2\eta_{k}\left(-\nu\right)\,x^{\frac{k}{2}+\nu-1}\,\left(x^{2}+y^{2}\right)^{-\frac{k}{2}-\nu},
\end{align*}
where the last step is just an application of the functional equation
of $\zeta_{k}(s)$, (\ref{functional equation zetak}). Analogously,
\[
\mathscr{R}_{3}(x)=-2\eta_{k}\left(-\nu\right)\,x^{\frac{k}{2}+\nu-1}
\]
and
\[
\mathscr{R}_{4}(x)=2\eta_{k}\left(\frac{k}{2}-\nu\right)\,x^{\frac{k}{2}-\nu-1}\,\left(x^{2}+y^{2}\right)^{-\frac{k}{2}+\nu}=2\eta_{k}\left(\nu\right)\,x^{\frac{k}{2}-\nu-1}\,\left(x^{2}+y^{2}\right)^{-\frac{k}{2}+\nu}.
\]
Using Cauchy's residue theorem and the calculations performed in the
previous lines, we find that (\ref{starting point of proof of guinand ryce rppopov})
implies the formula   
\begin{align}
2\Gamma\left(\frac{k}{2}\right)(\pi y)^{1-\frac{k}{2}}\sum_{m,n=1}^{\infty}r_{k}(m)\,r_{k}(n)\,\left(\frac{m}{n}\right)^{-\frac{\nu}{2}}\left(mn\right)^{\frac{1}{2}-\frac{k}{4}}J_{\frac{k}{2}-1}\left(2\pi\sqrt{mn}\,y\right)\,K_{\nu}\left(2\pi\sqrt{mn}\,x\right)\nonumber \\
=\frac{x^{1-\frac{k}{2}}}{4\pi i}\,\intop_{\frac{k}{2}-\sigma-i\infty}^{\frac{k}{2}-\sigma+i\infty}\eta_{k}\left(\frac{s-\nu-1}{2}+\frac{k}{4}\right)\,\eta_{k}\left(\frac{s+\nu-1}{2}+\frac{k}{4}\right)\nonumber \\
\times\,_{2}F_{1}\left(\frac{s-\nu-1}{2}+\frac{k}{4},\frac{s+\nu-1}{2}+\frac{k}{4};\,\frac{k}{2};-\frac{y^{2}}{x^{2}}\right)x^{-s}\,ds\nonumber \\
+\eta_{k}(\nu)\,x^{-\nu}\left\{ \left(x^{2}+y^{2}\right)^{\nu-\frac{k}{2}}-1\right\} +\eta_{k}(-\nu)\,x^{\nu}\left\{ \left(x^{2}+y^{2}\right)^{-\nu-\frac{k}{2}}-1\right\} ,\label{after the boring residues}
\end{align}
which is now in the right shape for us to perform the final transformations. Using the functional
equation for $\eta_{k}(s)$ and Euler's formula (\ref{Euler formula}),
the integral on the right-hand side of (\ref{after the boring residues})
can be written as 
\begin{align}
\intop_{\frac{k}{2}-\sigma-i\infty}^{\frac{k}{2}-\sigma+i\infty}\eta_{k}\left(\frac{s-\nu-1}{2}+\frac{k}{4}\right)\,\eta_{k}\left(\frac{s+\nu-1}{2}+\frac{k}{4}\right)\,_{2}F_{1}\left(\frac{s-\nu-1}{2}+\frac{k}{4},\frac{s+\nu-1}{2}+\frac{k}{4};\,\frac{k}{2};-\frac{y^{2}}{x^{2}}\right)x^{-s}\,ds\nonumber \\
=x^{-\frac{k}{2}}\,\left(\frac{x^{2}+y^{2}}{x^{2}}\right)^{1-\frac{k}{2}}\,\intop_{\sigma-i\infty}^{\sigma+i\infty}\eta_{k}\left(\frac{s+\nu+1}{2}\right)\,\eta_{k}\left(\frac{s-\nu+1}{2}\right)\,_{2}F_{1}\left(\frac{s+\nu+1}{2},\frac{s-\nu+1}{2};\,\frac{k}{2};-\frac{y^{2}}{x^{2}}\right)\left(\frac{x}{x^{2}+y^{2}}\right)^{-s}\,ds\nonumber \\
=\frac{8\pi i\,x^{\frac{k}{2}-1}}{x^{2}+y^{2}}\,\Gamma\left(\frac{k}{2}\right)(\pi y)^{1-\frac{k}{2}}\,\sum_{m,n=1}^{\infty}r_{k}(m)\,r_{k}(n)\,\left(\frac{m}{n}\right)^{-\frac{\nu}{2}}\left(m\,n\right)^{\frac{1}{2}-\frac{k}{4}}J_{\frac{k}{2}-1}\left(\frac{2\pi\sqrt{mn}\,y}{x^{2}+y^{2}}\right)\,K_{\nu}\left(\,\frac{2\pi\sqrt{mn}\,x}{x^{2}+y^{2}}\right),\label{at last after all that}
\end{align}
where in the last step we have used (\ref{starting point of proof of guinand ryce rppopov})
with $x$ and $y$ respectively replaced by $x/(x^{2}+y^{2})$ and
$y/(x^{2}+y^{2})$. Note that this last step is actually valid because,
by our initial choice of $\sigma$, we know that $\sigma>k-1+\text{Re}(\nu)$.
Using (\ref{at last after all that}) in (\ref{after the boring residues})
yields the desired identity (\ref{guinand ryce popov paper}).
\end{proof}

\textit{Acknowledgements:} This work was partially supported by CMUP, member of LASI, which is financed by national funds through FCT - Fundação para a Ciência e a Tecnologia, I.P., under the projects with reference UIDB/00144/2020 and UIDP/00144/2020. We also acknowledge the support from FCT (Portugal) through the PhD scholarship 2020.07359.BD. The author is grateful to Professor Maxim Korolev for noticing a typo in reference [13] of the previous version and for kindly sending articles about Popov's life and work. 
The author would also like to thank to Semyon Yakubovich for unwavering support and guidance throughout the writing of this paper.

\end{document}